\documentclass[12pt]{amsart}

\usepackage{amsfonts}
\usepackage{amsmath}
\usepackage{amssymb}
\usepackage{amsthm}
\usepackage{enumerate} 
\usepackage{hyperref}
\usepackage[margin=1.2in]{geometry}
\usepackage{graphicx}

\theoremstyle{plain}
\newtheorem{theorem}{Theorem}[section]
\newtheorem{lemma}[theorem]{Lemma}
\newtheorem{proposition}[theorem]{Proposition}

\theoremstyle{definition}

\newtheorem{remark}[theorem]{Remark}

\numberwithin{equation}{section}

\newcommand\N{\mathbb{N}}
\newcommand\Z{\mathbb{Z}}
\newcommand\R{\mathbb{R}}
\newcommand\Rp{\R_{+}}
\newcommand\C{\mathbb{C}}

\newcommand\D{\mathcal{D}}
\newcommand\E{\mathcal{E}}
\renewcommand\S{\mathcal{S}}
\renewcommand\O{\mathcal{O}}
\newcommand\OC{\O_{C}}
\newcommand\OM{\O_{M}}

\newcommand\ev[2]{\langle#1,#2\rangle}
\newcommand\compltens{\widehat{\otimes}}

\DeclareMathOperator\supp{supp}

\newcommand{\rotcong}{\rotatebox[origin=c]{-90}{$\cong$}}

\begin{document}

\title[Explicit sequence space representations on $\Rp$]{Explicit commutative sequence space representations of function and distribution spaces on the real half-line}
\author[A. Debrouwere]{Andreas Debrouwere}
\address{Department of Mathematics and Data Science \\ Vrije Universiteit Brussel, Belgium\\ Pleinlaan 2 \\ 1050 Brussels \\ Belgium}
\email{andreas.debrouwere@vub.be}

\author[L. Neyt]{Lenny Neyt}
\thanks{L. Neyt gratefully acknowledges support by the Research Foundation--Flanders through the postdoctoral grant 12ZG921N}
\address{Department of Mathematics: Analysis, Logic and Discrete Mathematics\\ Ghent University\\ Krijgslaan 281\\ 9000 Gent\\ Belgium}
\email{lenny.neyt@UGent.be}

\author[J. Vindas]{Jasson Vindas}
\thanks {J. Vindas was supported by Ghent University through the BOF-grant 01J04017 and by the Research Foundation--Flanders through the FWO-grant G067621N}
\address{Department of Mathematics: Analysis, Logic and Discrete Mathematics\\ Ghent University\\ Krijgslaan 281\\ 9000 Gent\\ Belgium}
\email{jasson.vindas@UGent.be}


\begin{abstract}
We provide explicit commutative sequence space representations for classical function and distribution spaces on the real half-line. This is done by evaluating at the Fourier transforms of the elements of an orthonormal wavelet basis. 
\end{abstract}
\maketitle

\section{Introduction: the commutative Valdivia-Vogt table on $\R_{+}$}

Sequence space representations of  spaces of smooth functions and distributions provide  great insight into their linear topological structure. Such isomorphic classification goes back to the works of Valdivia and Vogt in the early 1980s, where they independently discovered sequence space representations for most of the spaces used in Schwartz' distribution theory \cite{S-ThDistr}, see \cite{V-RepDK, V-SpDLp, V-RepOM, V-TopicLCS, V-SeqSpRepTestFuncDist}. The table of all these spaces and their sequence space representations was therefore called the Valdivia-Vogt structure table, and it has been recently completed by Bargetz \cite{B-RepOC, B-ComplVVTab}. The original proofs of many of the isomorphisms in this table are based on the Pelczi\'nsky decomposition method and are therefore  non-constructive. Hence, it is an interesting question to construct explicit  sequence space representations \cite{B-CommVVTab,B-Explicit, BDN-Explicit, OW-Explicit}. Furthermore, this often leads to sequence space representation tables  that may be interpreted as a commutative diagram \cite{B-CommVVTab, BDN-Explicit}.

In this article we provide the explicit construction of a  commutative sequence space representation table which includes many of the well-known spaces of smooth functions and distributions on $\Rp=(0,\infty)$. Besides delivering commutativity, our technique appears to be simpler than other methods considered in the literature. It uses the Fourier transform of a band-limited smooth orthonormal wavelet.

For the construction of our isomorphisms we consider a fixed even orthonormal wavelet $\eta$ for which its Fourier transform $\widehat{\eta}(x) = \int_{\R} \eta(t) e^{- 2 \pi i  x t} dt$ is a compactly supported real-valued smooth function (whose existence is guaranteed by a classical construction of Lemari\'{e} and Meyer \cite{H-W-FirstCourseWavelets}). We point out that $\widehat{\eta}$ must necessarily vanish in a neighborhood of the origin  \cite[Theorem 2.7, p. 108]{H-W-FirstCourseWavelets}. As $\eta_{j, k} = 2^{j / 2} \eta(2^{j} \cdot - k)$, $j, k \in \Z$, form an orthonormal basis for $L^{2}(\R)$, so do 
	\[ \widehat{\eta}_{j, k}(x) = 2^{- \frac{j}{2}} e^{-\frac{2 \pi i k}{2^{j}} x} 
	\widehat{\eta}
	\left(\frac{x}{2^{j}}\right), \qquad j,k \in \Z, \] 
by the Plancherel theorem. This implies that $\psi_{j, k} = \sqrt{2} \widehat{\eta}_{j, k|_{\Rp}}$, $j,k \in \Z$,
give rise to an orthonormal basis for $L^{2}(\Rp)$.

Let $\D(\Rp)$ denote the space of all smooth functions $\varphi$ on $\R$ such that $\supp \varphi \Subset \Rp$, endowed with its natural $(LF)$-space topology. We write $\D^{\prime}(\Rp)$ for the strong dual of $\D(\Rp)$, the space of distributions on the real half-line. As $\psi_{j, k} \in \D(\Rp)$ for any $j,  k \in \Z$, we find the following continuous linear mappings
	\[ \Psi_{j, k} : \D^{\prime}(\Rp) \rightarrow \C : \quad f \mapsto \ev{f}{\psi_{j, k}} . \]
These evaluations will enable us to establish explicit commutative sequence space representations,
 as described in the following theorem. 
All undefined spaces will be considered in Section \ref{s:Proof}, where a proof of Theorem \ref{t:CommSeqRepR+} will be given. 
We only remark here that the symbols $\bigoplus$ and $\Pi$ stand for  direct sums and Cartesian products of topological vector spaces \cite{Jarchow} and that we use standard notation from the theory of completed tensor products of topological vector spaces \cite{G-ProdTensTopEspNucl, Jarchow}.

	\begin{theorem}
		\label{t:CommSeqRepR+}
		The following table of isomorphisms holds
			\[ \begin{matrix}
				\D(\Rp) & \subset & \S(\Rp) & \subset & \OC(\Rp) & \subset & \OM(\Rp) & \subset & \E(\Rp) \\
				\rotcong & ~ & \rotcong &  ~ & \rotcong & ~ & \rotcong & ~ & \rotcong \\
				\bigoplus_{\Z} s & ~ & \kappa_1 \compltens s & ~ &  \kappa_1^{\prime}\compltens_{\iota} s  & ~ & \kappa_1^{\prime} \compltens s & ~ & \prod_{\Z} s \\ 
				~ & ~ & ~ & ~ & ~ & ~ & ~ & ~ & ~ \\
				\E^{\prime}(\Rp) & \subset & \OM^{\prime}(\Rp) & \subset & \OC^{\prime}(\Rp) & \subset & \S^{\prime}(\Rp) & \subset & \D^{\prime}(\Rp) \\ 
				\rotcong & ~ & \rotcong & ~ & \rotcong & ~ & \rotcong & ~ & \rotcong \\
				\bigoplus_{\Z} s^{\prime} & ~ & \kappa_1 \compltens_{\iota} s^{\prime} & ~ &  \kappa_1 \compltens s^{\prime} & ~ & \kappa_1^{\prime} \compltens s^{\prime} & ~ & \prod_{\Z} s^{\prime}
			\end{matrix} \]
and all the isomorphisms are induced by the restrictions of the mapping
			\begin{equation}
				\label{eq:Psi} 
				\Psi : \D^{\prime}(\Rp) \rightarrow \prod_{\Z} s^{\prime} : \quad f \mapsto \Psi(f) := ((\Psi_{j, k}(f))_{k \in \Z})_{j \in \Z} . 
			\end{equation}
	\end{theorem}

Theorem \ref{t:CommSeqRepR+}  implies that $\psi_{j, k}$, $j,k \in \Z$, form a common unconditional Schauder basis for all the spaces occurring  in the table of Theorem \ref{t:CommSeqRepR+}. 
Finally, we  remark that, by tensoring, we may extend Theorem \ref{t:CommSeqRepR+} to any open orthant $\Rp^{d} = (0, \infty)^{d}$. Moreover, by composing one obtains commutative sequence space representation tables for spaces of smooth functions and distributions on any open set $\Omega \subseteq \R^{d}$ that is globally $C^{\infty}$-diffeomorphic to $\Rp^{d}$. Particular examples are the whole space $\R^{d}$ (cf.\  \cite{BDN-Explicit}), open balls, and open bands. 
\section{The proof of Theorem \ref{t:CommSeqRepR+}}
\label{s:Proof}

In this section we verify the isomorphisms in the table of Theorem \ref{t:CommSeqRepR+} as restrictions of the mapping $\Psi$ in \eqref{eq:Psi}. However, due to the following lemma, it suffices just to show the isomorphisms for the smooth function spaces, as this automatically entails the isomorphisms for the dual spaces. 

	\begin{lemma}
		\label{l:IsoFuncSpImpliesIsoDual}
		Let $X$ be any of the smooth function spaces appearing in the table of Theorem \ref{t:CommSeqRepR+} and let $S$ be the corresponding sequence space. If $\Psi_{| X} : X \rightarrow S$ is an isomorphism, then so is $\Psi_{| X^{\prime}} : X^{\prime} \rightarrow S^{\prime}$. 
	\end{lemma}
	
	\begin{proof}
		Since $\Psi_{| X}$ is an isomorphism, so is $\Psi^{t}_{| X} : S^{\prime} \rightarrow X^{\prime}$. Now, for any $(c_{j, k})_{j, k \in \Z} \in S^{\prime} \subseteq \prod_{\Z} s^{\prime}$, as $\psi_{j, k} \in X$ for all $j, k \in \Z$, we have that $\ev{\Psi^{t}_{| X}((c_{j, k})_{j, k \in \Z})}{\psi_{j_{0}, k_{0}}} = \ev{(c_{j, k})_{j, k \in \Z}}{\Psi_{| X}(\psi_{j_0, k_0})} = c_{j_{0}, k_{0}}$ for any $j_{0}, k_{0} \in \Z$. This shows that $\Psi_{| X^{\prime}} = (\Psi^{t}_{| X})^{-1}$, whence $\Psi_{| X^{\prime}}$ is an isomorphism. 
	\end{proof}

\begin{remark}\label{remark-duality}
In connection with Lemma \ref{l:IsoFuncSpImpliesIsoDual}, we point out the following duality relations
\begin{gather*}
\big( \bigoplus_{\Z} s \big)' =  \prod_{\Z} s^{\prime}, \quad  (\kappa_1 \compltens s)' = \kappa'_1 \compltens s'  , \quad  (\kappa_1^{\prime}\compltens_{\iota} s)' =  \kappa_1\compltens s' \\
 ( \kappa_1^{\prime} \compltens s)' =  \kappa_1 \compltens_{\iota} s', \quad \big(\prod_{\Z} s \big)' = \bigoplus_{\Z} s',
\end{gather*}
where $\kappa_1$ and $\kappa'_1$ are defined in \eqref{exp-seq} and \eqref{exp-seq-dual} below. The first and last isomorphisms are obvious.  The second one follows from  \cite[Th\'eor\`eme 12, p.\ 76, Chapitre II]{G-ProdTensTopEspNucl}. The spaces  $\kappa_1^{\prime} \compltens s$ and $\kappa_1\compltens s'$ are bornological  \cite[Corollaire 1, p.\ 127, Chapitre II]{G-ProdTensTopEspNucl}. Hence,  the  fourth isomorphism follows from \cite[Corollaire, p.\ 90, Chapitre II]{G-ProdTensTopEspNucl}. By the same result, we have that $(\kappa_1\compltens s')' = \kappa_1^{\prime}\compltens_{\iota} s$, whence the third isomorphism is a consequence of the fact that the space $\kappa_1\compltens s'$ is reflexive (as it is bornological).
\end{remark}

We now perform some preliminary calculations that will be applied several times in what follows. Let us introduce some notation. We denote by $\E(\Rp)$ the space of all smooth functions on $\Rp$. Then, for any compact subset $K \Subset \Rp$ and any $n \in \N$, we consider the seminorm $\|\varphi\|_{C^{n}_{K}} = \max_{m \leq n} \sup_{x \in K} |\varphi^{(m)}(x)|$, $\varphi \in \E(\Rp)$. Also, given $n\in \Z$, the space $s^{n}$ stands for the Banach space of all sequences $(c_{k})_{k \in \Z} \in \C^{\Z}$ such that $\|(c_{k})_{k \in \Z}\|_{s^{n}} = \sup_{k \in \Z} (1 + |k|)^{n} |c_{k}| < \infty$. Then, 
	\[ s = \bigcap_{n \in \N} s^{n} \quad \text{ and } \quad s^{\prime} = \bigcup_{n \in \N} s^{-n} , \]  
endowed with their natural topologies. Set $\psi = \psi_{0,0} = \sqrt{2} \widehat{\eta}_{|_{\Rp}}$. We denote by $K_{\psi} = [R_{0}, R_{1}]$, $0 < R_{0} < R_{1}$, the smallest closed interval containing the support of $\psi$ in $\Rp$, and similarly $\operatorname*{supp} \psi_{j, k}\subseteq K_{\psi_{j, k}} = [2^{j} R_{0}, 2^{j} R_{1}]$ for any $j, k \in \Z$. We first find the following upper bound for $\Psi_{j, k}(\varphi)$ for any smooth function $\varphi$.

	\begin{lemma}
	\label{lemma1SR}
		For every $n \in \N$, there exists a $C_{n} > 0$ such that for any $\varphi \in \E(\Rp)$,
			\begin{equation}
				\label{eq:PsiCinfty}
				\left| \Psi_{j, k}(\varphi) \right| \leq C_{n} (1 + |k|)^{-n} \left( 2^{j / 2} \|\varphi(2^{j} \cdot)\|_{C^{n}_{K_{\psi}}} \right) , \qquad \forall j, k \in \Z. 
			\end{equation}
	\end{lemma}
	
	\begin{proof}
		For any $n \in \N$, we have for arbitrary $\varphi \in \E(\Rp)$ and $j,k \in \Z$,
			\begin{align*}
				\left| k^{n} \int_{0}^{\infty} \varphi(x) \psi_{j, k}(x) dx \right|
				&\leq \frac{2^{nj - j / 2}}{(2 \pi)^{n}} \sum_{m \leq n} {n \choose m} \frac{1}{2^{(n - m)j}} \int_{0}^{\infty} |\varphi^{(m)}(x)| |\psi^{(n - m)}\left(\frac{x}{2^{j}}\right)| dx \\
				&\leq \frac{2^{j / 2}}{(2 \pi)^{n}} \|\psi\|_{C^{n}_{K_{\psi}}} \sum_{m \leq n} {n \choose m} \int_{K_{\psi}} | \varphi(2^{j} x)^{(m)} | dx \\
				&\leq \frac{2^{n} (R_1-R_0) \|\psi\|_{C^{n}_{K_{\psi}}}}{(2 \pi)^{n}} \left( 2^{j / 2} \|\varphi(2^{j} \cdot)\|_{C^{n}_{K_{\psi}}} \right),
			\end{align*}
which yields the assertion. 

	\end{proof}
We also have the next bound.

	\begin{lemma}
		\label{l:PsiInvD}
		For any $n \in \N$, there exists a $C'_{n} > 0$ such that for all $(c_{k})_{k \in \Z} \in s^{n + 2}$,
			\begin{equation}
				\label{eq:PsiInvD}
				\max_{0 \leq  m \leq n}\sup_{x \in \Rp} \left| \frac{d^{m}}{dx^{m}} \sum_{k \in \Z} c_{k} \psi_{j, k}(x) \right| \leq C'_{n} \frac{\|(c_{k})_{k \in \Z}\|_{s^{n + 2}}}{2^{j/2} \min\{1, 2^{jn}\}} , \qquad \forall j \in \Z .
			\end{equation}
	\end{lemma}
	
	\begin{proof} For any $n \in \N$, and arbitrary $(c_{k})_{k \in \Z} \in s^{n+2}$ and $j \in \Z$, we have
	
			\[ \left| \frac{d^{n}}{dx^{n}} \sum_{k \in \Z} c_{k} \psi_{j, k}(x) \right| 
				\leq \sum_{k \in \Z} |c_{k}| |\psi_{j, k}^{(n)}(x)| 
				\leq  \frac{4^{n} \pi^{n + 2}}{3} \|\psi\|_{C^{n}_{K_{\psi}}} \frac{\|(c_{k})_{k \in \Z}\|_{s^{n + 2}}}{2^{j(n + 1/2)}} .
			\]
	
	\end{proof}
	
	We are now ready to establish the isomorphisms in the table of Theorem \ref{t:CommSeqRepR+}. We begin with the spaces $\E(\Rp)$ and $\E^{\prime}(\Rp)$. Note that by Lemma \ref{l:IsoFuncSpImpliesIsoDual} and Remark \ref{remark-duality} we only have to deal with the test function space, a fact we will not mention anymore for the other spaces.

	\begin{proposition}
		The mappings
			\[ \Psi : \E(\Rp) \rightarrow \prod_{\Z} s  \quad \text{ and } \quad \Psi : \E^{\prime}(\Rp) \rightarrow \bigoplus_{\Z} s^{\prime} \]
		are isomorphisms.
	\end{proposition}
	
	\begin{proof}
		The continuity of $\Psi : \E(\Rp) \rightarrow \prod_{\Z} s$ is a direct consequence of Lemma \ref{lemma1SR}. Let us prove that its inverse $\Psi^{-1}:((c_{j, k})_{k \in \Z})_{j \in \Z}\mapsto \sum_{j, k \in \Z} c_{j, k} \psi_{j, k}$ also maps $\prod_{\Z} s$ continuously into $\E(\Rp)$.
Given a compact subset $K$ in $\Rp$, there are only finitely many $j_{1}, \ldots, j_{m} \in \Z$ such that $K \cap K_{\psi_{j_{l}, k}} \neq \emptyset$, $1 \leq l \leq m$, for any $k \in \Z$. Hence, the sum $\sum_{j, k \in \Z} c_{j, k} \psi_{j, k}$ coincides on $K$ with the sum $\sum_{l = 1}^{m} \sum_{k \in \Z} c_{j_{l}, k} \psi_{j_{l}, k}$, which is smooth by Lemma \ref{l:PsiInvD}. As $K \Subset \Rp$ was arbitrary, we conclude that $\sum_{j, k \in \Z} c_{j, k} \psi_{j, k}$ is a smooth function on $\Rp$ and the continuity of 
$\Psi^{-1}$ follows  from Lemma \ref{l:PsiInvD} as well. \end{proof}
	
	We now verify the isomorphisms for $\D(\Rp)$ and $\D^{\prime}(\Rp)$.
	
	\begin{proposition}
		\label{p:SeqRepD}
		The mappings
			\[ \Psi : \D(\Rp) \rightarrow \bigoplus_{\Z} s \quad \text{ and } \quad \Psi : \D^{\prime}(\Rp) \rightarrow \prod_{\Z} s^{\prime}  \]
		are isomorphisms.
	\end{proposition}
	
	\begin{proof}
		We first note that $\Psi : \D(\Rp) \rightarrow \bigoplus_{\Z} s$ is well-defined and continuous. Indeed, take any $\varphi \in \D(\Rp)$ and let $K \Subset \Rp$ denote the compact support of $\varphi$. Then, by \eqref{eq:PsiCinfty} we have that $|\Psi_{j, k}(\varphi)| \leq C_{n} 2^{(n + 1/2) |j|} \|\varphi\|_{C^{n}_{K}} (1 + |k|)^{-n}$ for any $j, k \in \Z$. As there are only  finitely many $j \in \Z$ such that $K \cap K_{\psi_{j, k}} \neq \emptyset$, the claim follows. The continuity of $\Psi^{-1}:  \bigoplus_{\Z} s \rightarrow \mathcal{D}(\Rp)$ is a direct consequence of Lemma \ref{l:PsiInvD}.
		
	\end{proof}
	
Let us introduce the remaining spaces. We consider the Fr\'{e}chet space
	\[ \S(\Rp) = \{ \varphi \in \S(\R) : \supp \varphi \subseteq [0, \infty) \} , \]
and its dual $\S^{\prime}(\Rp)$. Note that $\S(\Rp)$ is a closed subspace of $\S(\R)$, the Schwartz space on the real line \cite{S-ThDistr}.
We define $\OM(\Rp)$ as  the space of all those $\varphi \in \E(\Rp)$ for which
	\[ \forall n \in \N ~ \exists \gamma > 0 : \quad \sup_{x \in \Rp} \left(\max\{x,1/x\}\right)^{-\gamma} |\varphi^{(n)}(x)| < \infty , \]
endowed with its natural $(PLB)$-space topology. One readily verifies that $\OM(\Rp)$ is the multiplier space of $\S(\Rp)$, that is, $f \in \OM(\Rp)$ if and only if $f \cdot \varphi \in \S(\Rp)$ for any $\varphi \in \S(\Rp)$.  The space $\OC(\Rp)$ is defined as the $(LF)$-space 
\begin{equation}
\label{eq def OC R+}
\OC(\Rp)= \varinjlim_{\lambda<0} \varprojlim_{n\in \N} X_{\lambda,n}
\end{equation}
where the Banach spaces $X_{\lambda,n}$, $\lambda\in \R,$ $n\in\mathbb{N},$ are given by\footnote{The differential operators $x^{n}D^{n}$ occurring in the definition of $X_{\lambda,n}$ are very natural in the context of the group $\R_{+}$ since they commute with all dilation operators.} 
\begin{equation}
\label{eq Banach aux X}
 X_{\lambda,n}= \{\varphi \in C^{n}(\Rp): \: \|\varphi\|_{X_{\lambda,n}}= \max_{0\leq m\leq n}\sup_{x\in\Rp} (\max\{x,1/x\})^{\lambda} x^{m}|\varphi^{(m)}(x)|   <\infty \}.
\end{equation}
Then $\OC^{\prime}(\Rp)$ is exactly the convolutor space of $\S(\Rp)$, that is, a distribution $f \in \S^{\prime}(\Rp)$ belongs to  $\OC^{\prime}(\Rp)$ if and only\footnote{This can be deduced via an exponential change of variables, under which the space $\mathcal{S}(\Rp)$ corresponds to the space $\mathcal{K}_{1}(\R)$ of exponentially rapidly decreasing smooth functions, whose (additive) convolutor space has predual $\OC(\mathcal{K}_{1})=\{\phi \in \mathcal{E}(\R):\: (\exists\gamma)(\forall n\in\N) \: (\sup_{x\in\R} e^{\gamma |x|} |\phi^{(n)}(x)| <\infty)\},$ see \cite[Section 6]{D-V2021}.} if $f \ast_{M} \varphi \in \S(\Rp)$ for each $\varphi\in\S(\Rp)$, where $\ast_{M}$ stands for Mellin convolution. 

We will work with the following Fr\'{e}chet space of rapidly exponentially decreasing sequences
	\begin{equation}
	\label{exp-seq}
	\kappa_{1} = \left\{ (c_{j})_{j \in \Z} \in \C^{\Z} : \sup_{j \in \Z} 2^{\gamma |j|} |c_{j}| < \infty \text{ for all } \gamma > 0 \right\} , 
	\end{equation}
and its dual space
	\begin{equation}
	\label{exp-seq-dual}
	 \kappa_{1}^{\prime} = \left\{ (c_{j})_{j \in \Z} \in \C^{\Z} : \sup_{j \in \Z} 2^{-\gamma |j|} |c_{j}| < \infty \text{ for some } \gamma > 0 \right\} . 
	 \end{equation}
We mention that $\kappa_{1} \cong s$ and $\kappa^{\prime}_{1} \cong s^{\prime}$.  Finally, we also need to introduce the auxiliary Banach spaces of double sequences
\begin{equation}
\label{eq Banach aux Y}
Y_{\lambda,n}= \left\{(c_{j, k})_{j, k \in \Z} \in \C^{\Z \times \Z}:\: \|(c_{j, k})_{j, k \in \Z}\|_{Y_{\lambda,n}}= \sup_{(j,k)\in \Z\times\Z} 2^{\lambda|j|} (1+|k|)^{n}|c_{j,k}|<\infty  \right\},
\end{equation}
where we let $\lambda\in\R$ and $n\in \N$.

\begin{remark}\label{remark-double}
The completed tensor product sequence spaces occurring in the table of Theorem \ref{t:CommSeqRepR+} may be explicitly described as double sequence spaces  by means of the Banach spaces \eqref{eq Banach aux Y} in the following way
\begin{gather*}
 \kappa_{1} \compltens s= \varprojlim_{\lambda>0,\: n\in\N} Y_{\lambda,n}, \quad \kappa_{1}^{\prime} \compltens_{\iota} s= \varinjlim_{\lambda<0} \varprojlim_{n\in \N} Y_{\lambda,n}, \quad \kappa_{1}^{\prime} \compltens s=\varprojlim_{n\in \N} \varinjlim_{\lambda<0} Y_{\lambda,n}, \\ \\
  \kappa'_{1} \compltens s'= \varinjlim_{\lambda < 0,\: n\in\N} Y_{\lambda,-n}, \quad \kappa_{1} \compltens s'= \varprojlim_{\lambda>0} \varinjlim_{n\in \N} Y_{\lambda,-n}, \quad \kappa_{1} \compltens_{\iota} s' = \varinjlim_{n\in \N} \varprojlim_{\lambda>0} Y_{\lambda,-n}.
\end{gather*}
These isomorphisms follow from general results about completed tensor products, completed tensor product representations of vector-valued sequence spaces \cite{G-ProdTensTopEspNucl, Jarchow}, and the duality relations given in Remark \ref{remark-duality} (see also \cite[Section 3]{B-RepOC}).
\end{remark}

	\begin{proposition}
		\label{p:IsoS&Os}
		The six mappings
			\begin{equation}
			\label {eq iso S} \Psi : \S(\Rp) \rightarrow \kappa_{1} \compltens s, \qquad \Psi : \S^{\prime}(\Rp) \rightarrow \kappa^{\prime}_{1} \compltens s^{\prime} ,
			\end{equation}
			
			\begin{equation*}
			 \Psi : \OC(\Rp) \rightarrow \kappa_{1}^{\prime} \compltens_{\iota} s, \qquad \Psi : \OC^{\prime}(\Rp) \rightarrow \kappa_{1} \compltens s^{\prime}, 
			\end{equation*}
and		
			\begin{equation*}
			\Psi : \OM(\Rp) \rightarrow \kappa_{1}^{\prime} \compltens s, \qquad  \Psi : \OM^{\prime}(\Rp) \rightarrow \kappa_{1} \compltens_{\iota} s^{\prime}
			\end{equation*}
		are isomorphisms.
	\end{proposition}
	
\begin{proof} That the first two mappings \eqref{eq iso S} are isomorphisms follows from \cite[Proposition 3.7]{Saneva-Vindas}, but the short proof we give here simultaneously applies to all cases. Besides being involved in \eqref{eq def OC R+}, the Banach spaces \eqref{eq Banach aux X} can also be used to describe $\S(\Rp)$ and $\OM(\Rp)$. In fact, clearly,
\[\mathcal{S}(\Rp)= \varprojlim_{\lambda>0,\: n\in\N} X_{\lambda,n} \quad \mbox{and} \quad 
\OM(\Rp)= \varprojlim_{n\in \N} \varinjlim_{\lambda<0} X_{\lambda,n}.
\]
The result is then an immediate consequence of Remark \ref{remark-double} and Lemma \ref{lemma iso XY} shown below.
\end{proof}
\begin{lemma}
\label{lemma iso XY} Let $\lambda\in\R$ and $n\in \N$. The mappings 
\[
\Psi: X_{\lambda,n}\to Y_{\lambda-1/2,n} \quad \mbox{and} \quad \Psi^{-1}: Y_{\lambda,n+2}\to X_{\lambda-1,n}
\]
are well-defined and continuous.
\end{lemma}
\begin{proof} First note that the function $\omega_{\lambda}(x)=(\max\{x,1/x\})^{\lambda}$ satisfies 
\begin{equation}
\label{eq semi multi weight}
\omega_{\lambda}(x y)\leq \omega_{\lambda}(x)\omega_{|\lambda|}(y), \qquad \forall x,y>0.
\end{equation}
Using this inequality and \eqref{eq:PsiCinfty}, we obtain for all $\varphi \in X_{\lambda,n}$
\begin{align*}
\left\| \Psi(\varphi) \right\|_{Y_{\lambda-1/2, n}}&\leq C_n \left\| \varphi \right\|_{X_{\lambda,n}}   \sup_{x\in K_{\psi},\: j\in\Z} 2^{j/2}\omega_{\lambda-1/2} (2^{j}) \frac{ \omega_{-\lambda}(2^j x)}{\min\{1, x^{n} \}}
\\
&\leq C_n  \left\| \varphi \right\|_{X_{\lambda,n}}   \sup_{x\in K_{\psi}}  \frac{ \omega_{|\lambda|}(x)}{\min\{1, x^{n} \}}.
\end{align*}
Proceeding as in the proof of Lemma \ref{l:PsiInvD} and using again \eqref{eq semi multi weight}, we have for all $(c_{j, k})_{j, k \in \Z} \in Y_{\lambda,n+2}$
\begin{align*}
& \left\| \Psi^{-1}((c_{j, k})_{j, k \in \Z}) \right\|_{X_{\lambda-1, n}} \\
&\leq \frac{4^{n} \pi^{n + 2}}{3} \left(\sup_{x\in K_{\psi}} \max\{1,x^{n}\}\right)\|\psi\|_{X_{|\lambda-1|,n}} \sum_{j\in\Z} 2^{-j/2}\omega_{\lambda-1}(2^j)  \sup_{k\in \Z}(1+|k|)^{n+2}|c_{j,k}|
\\
&
\leq  \frac{4^{n} \pi^{n + 2}}{3}  \left(\sup_{x\in K_{\psi}} \max\{1,x^{n}\}\right) \left(\sum_{j \in \Z} 2^{-|j|/2}\right)  \|\psi\|_{X_{|\lambda-1|,n}} \| (c_{j, k})_{j, k \in \Z}\|_{Y_{\lambda, n+2}}.
\end{align*}
\end{proof}

\end{document}